\documentclass[12pt,a4paper]{amsart}
\usepackage[width=0.75\paperwidth,height=0.85\paperheight,twoside=false,includehead,includefoot]{geometry}
\allowdisplaybreaks[1]

\usepackage{amstext}
\usepackage{amsthm}
\usepackage{amssymb}

\newtheorem{thm}{Theorem}[section]
\newtheorem{lem}[thm]{Lemma}
\newtheorem{prop}[thm]{Proposition}

\theoremstyle{definition}
\newtheorem{defn}[thm]{Definition}
\newtheorem{ex}[thm]{Example}
\theoremstyle{remark}
\newtheorem{rem}[thm]{Remark}

\begin{document}

\title{Sum formula for multiple zeta function}

\author{Minoru Hirose}
\address[Minoru Hirose]{Institute for Advanced Research, Nagoya University,  Furo-cho, Chikusa-ku, Nagoya, 464-8602, Japan}
\email{minoru.hirose@math.nagoya-u.ac.jp}

\author{Hideki Murahara}
\address[Hideki Murahara]{The University of Kitakyushu,  4-2-1 Kitagata, Kokuraminami-ku, Kitakyushu, Fukuoka, 802-8577, Japan}
\email{hmurahara@mathformula.page}

\author{Tomokazu Onozuka}
\address[Tomokazu Onozuka]{Institute of Mathematics for Industry, Kyushu University 744, Motooka, Nishi-ku,
Fukuoka, 819-0395, Japan}
\email{t-onozuka@imi.kyushu-u.ac.jp}


\begin{abstract}
 The sum formula is a well known relation in the field of the multiple
 zeta values. In this paper, we present its generalization for the
 Euler-Zagier multiple zeta function.
\end{abstract}

\maketitle

\section{Introduction}
The Euler-Zagier multiple zeta function (MZF) is defined by
\begin{align*}
 \zeta(s_{1},\dots,s_{r}):=\sum_{1\le n_{1}<\cdots<n_{r}}\frac{1}{n_{1}^{s_{1}}\cdots n_{r}^{s_{r}}},
\end{align*}
where $s_{i}\in\mathbb{C}\,(i=1,\ldots,r)$ are complex variables.
The number of the variables $r$ is called the depth of $\zeta(s_{1},\dots,s_{r})$.
Matsumoto \cite{Mat02} proved that the series is absolutely convergent
in the domain
\[
 \{(s_{1},\ldots,s_{r})\in\mathbb{C}^{r}\;|\;\Re(s_{l}+\cdots+s_{r})>r-l+1\;\; (l=1,\dots,r)\}.
\]
Akiyama, Egami, and Tanigawa
\cite{AET01} and Zhao \cite{Zha00} independently proved that $\zeta(s_{1},\dots,s_{r})$
is meromorphically continued to the whole space $\mathbb{C}^{r}$.
Furthermore, all possible poles of $\zeta(s_{1},\dots,s_{r})$ are located on $s_{l}+\cdots+s_{r}\in\mathbb{Z}_{\le r-l+1}$ for $l=1,\dots,r$.
Note that some possible poles are known not to be actual poles,
but we do not use this fact in this paper (for details, see \cite{AET01}).

The special values $\zeta(k_{1},\dots,k_{r})$ with $k_1,\dots,k_{r-1}\in\mathbb{Z}_{\ge1}$ and $k_{r}\in\mathbb{Z}_{\ge2}$ are called  multiple zeta values (MZVs).
The MZVs are real numbers and known to satisfy many kinds
of algebraic relations over $\mathbb{Q}.$
One of the most fundamental relations is the sum formula:
\begin{prop}[Sum formula; Granville \cite{Gra97}, Zagier] \label{sum}
 For positive integers $k,r$ with $k>r$, we have
 \begin{align*}
  \sum_{\substack{ k_{1}+\dots+k_{r}=k \\ k_1,\dots,k_{r-1}\ge1,k_{r}\ge2 }}
  \zeta(k_{1},\dots,k_{r}) & =\zeta(k).
 \end{align*}
\end{prop}

From the analytic point of view, Matsumoto \cite{Mat06} raised the
question whether the known relations among MZVs are valid only
for positive integers or not.
It is known that the harmonic relations,
e.g., $\zeta(s_{1})\zeta(s_{2})=\zeta(s_{1},s_{2})+\zeta(s_{2},s_{1})+\zeta(s_{1}+s_{2})$
are valid not only for positive integers but for complex numbers.
Matsumoto and Tsumura \cite[Proposition 2.1]{MT05} gave a relation
which is a generalization of Proposition \ref{sum} with $r=2$.
This relation consists not only of the MZFs but also of the Mordell-Tornheim multiple zeta functions.
Based on such circumstances, we give a generalization of Proposition \ref{sum} for the MZF.
\begin{thm} \label{main1}
 For $s\in\mathbb{C}$ with $\Re(s)>1$ and $s\ne2$, we have
 \begin{align*}
  \sum_{n=0}^{\infty}(\zeta(s-n-2,n+2)-\zeta(-n,s+n)) & =\zeta(s).
 \end{align*}
\end{thm}
Theorem \ref{main1} is a generalization of Proposition \ref{sum} with depth $r=2$.
This can be generalized to the arbitrary depth (see Theorem \ref{main2}).
\begin{defn}
 For a non-negative integer $a$ and a positive integer $b$, we define
 $G_{a,b}(s_{1},\dots,s_{a};s)$ inductively by
 \begin{align*}
  G_{a,1}(s_{1},\dots,s_{a};s)
  & :=\zeta(s_{1},\dots,s_{a},s),\\
  G_{a,b}(s_{1},\dots,s_{a};s)
  & :=\sum_{n=0}^{\infty} G_{a+1,b-1}(s_{1},\dots,s_{a},s-n-b;n+b) \\
  & \quad -\sum_{n=0}^{\infty} G_{a+1,b-1}(s_{1},\dots,s_{a},-n;s+n),
 \end{align*}
 where
 $s_{1},\dots,s_{a},s$ are complex numbers such that
 $\Re(s)>b, \Re(s+s_a)>1+b,\dots,\Re(s+s_a+\cdots+s_1)>a+b$.
\end{defn}
\begin{rem}
 The convergence will be proved in Lemma \ref{lem:G_ab_formula}.
\end{rem}
\begin{ex}
 We show some examples of $G_{a,b}(s_{1},\dots,s_{a};s)$:
 \begin{align*}
  &G_{a,2}(s_{1},\dots,s_{a};s) \\
  &=\sum_{n=0}^{\infty} \zeta(s_{1},\dots,s_{a},s-n-2,n+2)
   -\sum_{n=0}^{\infty} \zeta(s_{1},\dots,s_{a},-n,s+n), \\
  &G_{a,3}(s_{1},\dots,s_{a};s) \\
  &=\sum_{n_1=0}^{\infty}
   \Biggl(
   \sum_{n_2=0}^{\infty}
   \zeta(s_{1},\dots,s_{a},s-n_1-3,n_1-n_2+1,n_2+2) \\
  &\qquad\quad\,\,\,
   -\sum_{n_2=0}^{\infty}  \zeta(s_{1},\dots,s_{a},s-n_1-3,-n_2,n_1+n_2+3)
   \Biggr)\\
  &\quad -\sum_{n_1=0}^{\infty}
   \Biggl(
   \sum_{n_2=0}^{\infty} \zeta(s_{1},\dots,s_{a},-n_1,s+n_1-n_2-2,n_2+2) \\
  &\qquad\qquad\,\,
   -\sum_{n_2=0}^{\infty}  \zeta(s_{1},\dots,s_{a},-n_1,-n_2,s+n_1+n_2)
   \Biggr).
 \end{align*}
\end{ex}
\begin{rem}
 We note that $G_{a,b}(s_{1},\dots,s_{a};s)$ is a sum of MZFs of depth $a+b$.
\end{rem}
\begin{thm} \label{main2}
 Let $b$ be a positive integer.
 For $s\in\mathbb{C}$ with $\Re(s)>b$, we have
 \[
  G_{0,b}(s)=\zeta(s).
 \]
\end{thm}

\begin{rem}
 If $s\in\mathbb{Z}_{>b}$, $G_{a,b}(s_{1},\dots,s_{a};s)$ is equal to
 \[
  \sum_{\substack{m_{1}+\cdots+m_{b}=s\\
  m_{1},\dots,m_{b-1}\ge1\\
  m_{b}\ge2
  }
  }\zeta(s_{1},\dots,s_{a},m_{1},\dots,m_{b})
 \]
 by definition.
 Thus the case $s\in\mathbb{Z}_{>b}$ of the theorem implies Proposition \ref{sum}.
\end{rem}

\section{Proof of theorem \ref{main1} }
Let $\sigma:=\Re(s)$ and $\sigma_i:=\Re(s_i)$ for $i\in\mathbb{Z}_{\ge1}$.
The proof of Theorem \ref{main1} is divided into two parts.
For $\sigma>2$, we use series transformation.
For $\sigma>1$, we prove the theorem by showing the regularness of the infinite series on the left-hand side and using analytic continuation.

\begin{proof}[Proof of Theorem \ref{main1} for $\sigma>2$]
 Since the functions $\zeta(s-n-2,n+2)$ and $\zeta(-n,s+n)$ are absolutely
 convergent in $\sigma>2$, we have
 \begin{align*} 
  \begin{split}
   & \sum_{n=0}^{\infty}(\zeta (s-n-2,n+2)-\zeta (-n,s+n))\\
   & =\sum_{n=0}^{\infty}\sum_{0<m_{1}<m_{2}}\left(\frac{1}{m_{1}^{s-n-2}m_{2}^{n+2}}-\frac{1}{m_{1}^{-n}m_{2}^{s+n}}\right)\\
   & =\sum_{n=0}^{\infty}\sum_{0<m_{1}<m_{2}}\left(\frac{m_{1}}{m_{2}}\right)^{n}\left(\frac{1}{m_{1}^{s-2}m_2^2}-\frac{1}{m_{2}^{s}}\right)\\
   & =\sum_{0<m_{1}<m_{2}}\frac{1}{m_{2}^{2}-m_{1}m_{2}}\left(\frac{1}{m_{1}^{s-2}}-\frac{1}{m_{2}^{s-2}}\right).
  \end{split}
 \end{align*}
 Here we note that interchanging the order of summations $\sum_{n=0}^{\infty}$ and $\sum_{0<m_{1}<m_{2}}$ is valid
 because the last sum is absolutely convergent in $\sigma>2$.
 Since
 \begin{align*}
  \sum_{0<m_{1}<m_{2}}\frac{1}{m_{2}^{2}-m_{1}m_{2}}\frac{1}{m_{1}^{s-2}}
  & =\sum_{0<m_{1}}\frac{1}{m_{1}^{s-2}}\sum_{m_{1}<m_{2}}\frac{1}{m_{2}^{2}-m_{1}m_{2}}\\
  & =\sum_{0<m_{1}}\frac{1}{m_{1}^{s-2}}\sum_{m_{1}<m_{2}}\frac{1}{m_{1}}\left(\frac{1}{m_{2}-m_{1}}-\frac{1}{m_{2}}\right)\\
  & =\sum_{0<m_{1}}\frac{1}{m_{1}^{s-1}}\left(1+\frac{1}{2}+\cdots+\frac{1}{m_{1}}\right)
 \end{align*}
 and
 \begin{align*}
  \sum_{0<m_{1}<m_{2}}\frac{1}{m_{2}^{2}-m_{1}m_{2}}\frac{1}{m_{2}^{s-2}}
  & =\sum_{1<m_{2}} \frac{1}{m_{2}^{s-1}} \sum_{0<m_{1}<m_{2}} \frac{1}{m_{2}-m_{1}}\\
  & =\sum_{0<m_{2}}\frac{1}{m_{2}^{s-1}}\left(1+\frac{1}{2}+\cdots+\frac{1}{m_{2}-1}\right),
 \end{align*}
 we have
 \begin{align*}
  \sum_{n=0}^{\infty}(\zeta (s-n-2,n+2)-\zeta (-n,s+n))
  =\sum_{0<m}\frac{1}{m^{s}}
  =\zeta(s).
 \end{align*}
 Then we find the result.
\end{proof}

Next, we prove Theorem \ref{main1} for $1<\sigma\le 2$.
\begin{lem}
 For $s\in\mathbb{C}$ with $s\ne2$ and $\sigma>1$, and $\alpha\in\mathbb{C}$ with $\Re(\alpha)>1$,
 the analytic continuation of $\zeta(s-\alpha,\alpha)$ can be given by
 \[
  \zeta(s-\alpha,\alpha)
  =\sum_{0<m_1} \frac{ 1 }{ m_1^{s-\alpha} }
   \sum_{m_1<m_2} \biggl( \frac{ 1 }{ m_2^{\alpha} } -\int_0^1 \frac{ dt }{ (m_2-t)^{\alpha} } \biggr)
   +\frac{ \zeta(s-1) }{ \alpha-1 }.
 \]
\end{lem}
\begin{proof}
 Since
 \[
  \frac{ 1 }{ m_2^{\alpha} } -\int_0^1 \frac{ dt }{ (m_2-t)^{\alpha} }
  =O\biggl( \frac{ 1 }{ m_2^{\alpha+1} } \biggr)
  \qquad (m_2\rightarrow \infty),
 \]
 the first term on the right-hand side converges absolutely for $\sigma>1$.
 Furthermore, if $\sigma>2$, we have
 \begin{align*}
  &\sum_{0<m_1} \frac{ 1 }{ m_1^{s-\alpha} }
   \sum_{m_1<m_2} \biggl( \frac{ 1 }{ m_2^{\alpha} } -\int_0^1 \frac{ dt }{ (m_2-t)^{\alpha} } \biggr) \\
  &=\sum_{0<m_1} \frac{ 1 }{ m_1^{s-\alpha} }
   \biggl( \sum_{m_1<m_2} \frac{ 1 }{ m_2^{\alpha} } -\int_{m_1}^{\infty} \frac{ du }{ u^{\alpha} } \biggr) \\
  &=\zeta(s-\alpha,\alpha) -\frac{ 1 }{ \alpha-1 } \zeta(s-1). \qedhere
 \end{align*}
\end{proof}

We denote by $[t]$ the greatest integer less than or equal to $t$.
\begin{lem}
 For $s\in\mathbb{C}$ with $s\ne2$ and $\sigma>1$, and $\alpha\in\mathbb{C}$ with $\Re(\alpha)>1$, we have
 \[
  \zeta(s-\alpha,\alpha)
  =-\sum_{0<m} \frac{ \alpha }{ m^{s+1} } \int_{0}^{\infty} \biggl( \frac{ m }{ m+t } \biggr)^{\alpha+1} (t-[t]) dt
   +\frac{ \zeta(s-1) }{ \alpha-1 }.
 \]
\end{lem}

\begin{proof}
 We have
 \begin{align*}
  \frac{1}{m_{2}^{\alpha}}-\int_{0}^{1} \frac{ dt }{ \left(m_{2}-t\right)^{\alpha} }
  &=\int_{0}^{1}\left(\frac{1}{m_{2}^{\alpha}}-\frac{1}{\left(m_{2}-t\right)^{\alpha}}\right) dt \\
  &=-\int_{0<u<t<1} \frac{\alpha}{\left(m_{2}-u\right)^{\alpha+1}} du dt \\
  &=-\int_{0}^{1} \frac{\alpha(1-u)}{\left(m_{2}-u\right)^{\alpha+1}} du.
 \end{align*}
 Thus we have
 \begin{align*}
  \zeta(s-\alpha,\alpha)
  &=-\sum_{0<m_{1}} \frac{1}{m_{1}^{s-\alpha}}
   \sum_{m_{1}<m_{2}} \int_{0}^{1} \frac{\alpha(1-t)}{\left(m_{2}-t\right)^{\alpha+1}} dt
   +\frac{\zeta(s-1)}{\alpha-1} \\
  &=-\sum_{0<m_{1}} \frac{\alpha}{m_{1}^{s+1}}
   \sum_{m_{1}<m_{2}} \int_{0}^{1}\left(\frac{m_{1}}{m_{2}-t}\right)^{\alpha+1}(1-t) dt
   +\frac{\zeta(s-1)}{\alpha-1} \\
  &=-\sum_{0<m_{1}} \frac{\alpha}{m_{1}^{s+1}}
   \sum_{m_{1}\le m_{2}} \int_{0}^{1}\left(\frac{m_{1}}{m_{2}+t}\right)^{\alpha+1} t dt
   +\frac{\zeta(s-1)}{\alpha-1} \\
  &=-\sum_{0<m} \frac{\alpha}{m^{s+1}}
   \int_{0}^{\infty}\left(\frac{m}{m+t}\right)^{\alpha+1}(t-[t]) dt
   +\frac{\zeta(s-1)}{\alpha-1}.
 \end{align*}
 This finishes the proof.
\end{proof}
From the previous lemma, we have
\[
 \frac{ \partial }{ \partial\alpha } \zeta(s-\alpha,\alpha)
 =-H_{1}(\alpha) +\alpha H_{2}(\alpha) -H_{3}(\alpha),
\]
where
\begin{align*}
 H_{1}(\alpha)
 &=\sum_{0<m} \frac{ 1 }{ m^{s+1} } \int_{0}^{\infty} \biggl( \frac{ m }{ m+t } \biggr)^{\alpha+1}(t-[t])dt, \\
 H_{2}(\alpha)
 &=\sum_{0<m} \frac{ 1 }{ m^{s+1} } \int_{0}^{\infty} \log\biggl( \frac{ m+t }{ m } \biggr) \biggl( \frac{ m }{ m+t } \biggr)^{\alpha+1}(t-[t])dt, \\
 H_{3}(\alpha)
 &=\frac{ \zeta(s-1) }{ (\alpha-1)^2 }.
\end{align*}

\begin{lem} \label{ooooo1}
 Assume that $s \in \mathbb{R}$ and $1<s<2$.
 Then, for $\alpha \in \mathbb{R}_{>1}$, we have
 \begin{align*}
  H_{1}(\alpha)=O\biggl(\frac{1}{\alpha^{s}}\biggr) \qquad(\alpha \rightarrow \infty).
 \end{align*}
\end{lem}
\begin{proof}
 Since
 \begin{align*}
  \left(\frac{m}{m+t}\right)^{\alpha+1} \le\biggl(\frac{z}{z+t}\biggr)^{\alpha+1}, \\
  \frac{1}{m^{s+1}} \le \frac{2^{s+1}}{(m+1)^{s+1}} \le \frac{2^{s+1}}{z^{s+1}}
 \end{align*}
 for $t>0, m\in\mathbb{Z}_{>0}$, and $m \le z \le m+1$, we have
 \begin{align*}
  H_{1}(\alpha)
  &\le 2^{s+1} \int_{1}^{\infty}\biggl(\frac{1}{z^{s}} \int_{0}^{\infty}\biggl(\frac{z}{z+t}\biggr)^{\alpha+1}(t-[t])dt\biggr) \frac{dz}{z} \\
  &\le 2^{s+1} \int_{0}^{\infty}\biggl(\frac{1}{z^{s}} \int_{0}^{\infty}\biggl(\frac{z}{z+t}\biggr)^{\alpha+1}(t-[t])dt \biggr) \frac{dz}{z} \\
  &=2^{s+1} \int_{0}^{\infty} \frac{z^{\alpha-s}}{(1+z)^{\alpha+1}} dz \int_{0}^{\infty} \frac{t-[t]}{t^{s}} dt\\
  &=2^{s+1} B(\alpha+1-s, s) \int_{0}^{\infty} \frac{t-[t]}{t^{s}} dt,
 \end{align*}
 where $B$ is the beta function.
 Since the integral
 \[
  \int_{0}^{\infty} \frac{t-[t]}{t^{s}} dt
 \]
 is convergent and
 \[
  B(\alpha+1-s, s)
  =O\biggl(\frac{1}{\alpha^{s}}\biggr),
 \]
 we have the result.
\end{proof}

\begin{lem} \label{ooooo2}
 Assume that $s \in \mathbb{R}$ and $1<s<2$.
 Then, for $\alpha \in \mathbb{R}_{>1}$, we have
 \[
  H_{2}(\alpha)=O\biggl(\frac{1}{\alpha^{s+1}}\biggr) \qquad(\alpha \rightarrow \infty).
 \]
\end{lem}
\begin{proof}
 Since $\log x \le x-1$ for $x \geq 1$, we have
 \begin{align*}
  H_{2}(\alpha)
  &\le \sum_{0<m} \frac{ 1 }{m^{s+1}} \int_{0}^{\infty}\left(\frac{m+t}{m}-1\right)\left(\frac{m}{m+t}\right)^{\alpha+1}(t-[t]) dt \\
  &=\sum_{0<m} \frac{ 1 }{m^{s+2}} \int_{0}^{\infty}\left(\frac{m}{m+t}\right)^{\alpha+1} t(t-[t]) dt.
 \end{align*}
 Since
 \begin{align*}
   \left(\frac{m}{m+t}\right)^{\alpha+1} \le\left(\frac{z}{z+t}\right)^{\alpha+1}, \\
   \frac{1}{m^{s+2}} \le \frac{2^{s+2}}{(m+1)^{s+2}} \le \frac{2^{s+2}}{z^{s+2}}
 \end{align*}
 for any $0<t, m \in \mathbb{Z}_{>0}$ and $m \le z \le m+1$, we have
 \begin{align*}
  \frac{ H_{2}(\alpha) }{ 2^{s+2} }
  &\le \frac{ 1 }{ 2^{s+2} } \sum_{0<m} \frac{ 1 }{m^{s+2}} \int_{0}^{\infty}\left(\frac{m}{m+t}\right)^{\alpha+1} t(t-[t]) dt \\
  &\le \int_{1}^{\infty} \biggl( \frac{ 1 }{z^{s+2}} \int_{0}^{\infty} \biggl(\frac{ z }{ z+t }\biggr)^{\alpha+1} t(t-[t])dt \biggr) dz \\
  &\le \int_{0}^{\infty} \biggl( \frac{ 1 }{z^{s+1}} \int_{0}^{\infty} \biggl(\frac{ z }{ z+t }\biggr)^{\alpha+1} t(t-[t])dt \biggr) \frac{ dz }{ z }  \\
  &=\int_{0}^{\infty} \frac{ z^{\alpha-s-1} }{ (1+z)^{\alpha+1} } dz
   \int_{0}^{\infty} \frac{t-[t]}{t^{s}} dt \\
  &=B(\alpha-s,s+1) \int_{0}^{\infty} \frac{t-[t]}{t^{s}} dt.
 \end{align*}
 Similar to the proof of the previous lemma, we have the result.
\end{proof}

\begin{lem} \label{aaaallll}
 Let $M>0$.
 For $0<\delta<\min(1,\sigma-1)\in\mathbb{R}$, we have
 \[
  \frac{\partial}{\partial \alpha} \zeta(s-\alpha, \alpha)
  =O\biggl(\frac{1}{\Re(\alpha)^{1+\delta}}\biggr)
  \qquad(\Re(\alpha) \rightarrow \infty, |\Im(\alpha)|<M).
 \]
\end{lem}
\begin{proof}
 %
 Since
 \[
  H_{3}(\alpha)=O\left(\frac{1}{\Re(\alpha)^{2}}\right)
  \qquad (\Re(\alpha) \rightarrow \infty),
 \]
 we need to show
 \begin{align} \label{ppppp1}
  H_{1}(\alpha)=O\left(\frac{1}{\Re(\alpha)^{1+\delta}}\right)
   \qquad(\Re(\alpha) \rightarrow \infty)
 \end{align}
 and
 \begin{align} \label{ppppp2}
  \alpha H_{2}(\alpha)=O\left(\frac{1}{\Re(\alpha)^{1+\delta}}\right)
   \qquad(\Re(\alpha) \rightarrow \infty, |\Im(\alpha)|<M).
 \end{align}
 Note that
 \begin{align*}
  \left|H_{1}(\alpha)\right|
  &\le \sum_{0<m} \frac{1}{m^{(1+\delta)+1}} \int_{0}^{\infty}\left(\frac{m}{m+t}\right)^{\Re(\alpha)+1}(t-[t]) dt, \\
  \left|H_{2}(\alpha)\right|
  &\le \sum_{0<m} \frac{ 1 }{m^{(1+\delta)+1}}
   \int_{0}^{\infty} \log \left(\frac{m+t}{m}\right)\left(\frac{m}{m+t}\right)^{\Re(\alpha)+1}(t-[t]) dt
 \end{align*}
 hold.
 Then, by Lemmas \ref{ooooo1} and \ref{ooooo2}, we obtain \eqref{ppppp1} and \eqref{ppppp2}.
\end{proof}

\begin{prop} \label{abcd}
 Let $M>0$ and $s\in\mathbb{C}$ with $s\ne2$, $\sigma>1$, and $|\Im(s)|<M$.
 For $0<\delta<\min(1,\sigma-1)\in\mathbb{R}$, we have
 \[
  |\zeta(s-n-2,n+2)-\zeta(-n,s+n)|\le O\biggl( \frac{ 1 }{ n^{1+\delta} } \biggr).
 \]
\end{prop}
\begin{proof}
 From
 \begin{align*}
  \zeta(s-n-2, n+2)-\zeta(-n, s+n)
  &=\int_{n+s}^{n+2}\biggl(\frac{\partial}{\partial \alpha}\zeta(s-\alpha, \alpha)\biggr) d\alpha \\
  &=\int_{n+s}^{n+2} O\biggl( \frac{ 1 }{ \Re(\alpha)^{1+\delta} }  \biggr) d\alpha \\
  &=\int_{n+s}^{n+2} O\biggl( \frac{ 1 }{ \min(n+2,n+\sigma)^{1+\delta} }  \biggr) d\alpha
 \end{align*}
 by Lemma \ref{aaaallll}, we find the result.
\end{proof}
\begin{proof}[Proof of Theorem \ref{main1}]
 It follows from Proposition \ref{abcd} that
 the sum $\sum_{n=0}^{\infty}(\zeta (s-n-2,n+2)-\zeta (-n,s+n))$
 uniformly converges on any compact subsets of $\{s\in\mathbb{C} \mid s\ne2, \sigma>1 \}$, and holomorphic on the region.
 By the identity theorem, this finishes the proof.
\end{proof}

\section{Proof of Theorem \ref{main2}}
\begin{defn}
 For a positive integer $d$, a non-negative integer $D$, and $s\in\mathbb{C}$, we define
 \begin{align*}
  F_{d}(D;s)
  &:=\sum_{D<m}
   \frac{1}{m^{s-d}}
   \sum_{m-D\le x_{1}\le\cdots\le x_{d}\le m}\frac{1}{x_{1}\cdots x_{d}}, \\
  F^{(1)}_{d}(D;s)
  &:=\sum_{D<t<m}
   \frac{1}{t^{s-d-1} (m-t)}
   \sum_{m-t\le x_{1}\le\cdots\le x_{d}\le m}\frac{1}{x_{1}\cdots x_{d}}, \\
  F^{(2)}_{d}(D;s)
  &:=\sum_{D<t<m} \frac{1}{t^{s-d-1} m}
   \sum_{m-t\le x_{1}\le\cdots\le x_{d}\le m}\frac{1}{x_{1}\cdots x_{d}}, \\
  F^{(3)}_{d}(D;s)
  &:=\sum_{D<t<m} \frac{1}{m^{s-d-1}(m-t)}
   \sum_{m-t\le x_{1}\le\cdots\le x_{d}\le m}\frac{1}{x_{1}\cdots x_{d}}.
 \end{align*}
\end{defn}
\begin{lem} \label{3232}
 If $\sigma>1$, the function $F_{d}(D;s)$ converges absolutely.
 In addition, if $\sigma>d+2$, the functions $F^{(1)}_{d}(D;s)$, $F^{(2)}_{d}(D;s)$, and $F^{(3)}_{d}(D;s)$ converge absolutely.
 Moreover, we have

 \begin{align}\label{Fdi_upper}
  \Bigl|F^{(i)}_{d}(D;s) \Bigr|
  \ll \sum_{ D<t } \frac{ (\log t)^{d+1} }{t^{\sigma-d-1} }
 \end{align}
 for $i=1,2,3$, where the implicit constant does not depend on $D$ and $s$.
\end{lem}
\begin{proof}
 The convergence of $F_{d}(D;s)$ is immediate. The convergence of $F^{(i)}_{d}(D;s)$ follows from (\ref{Fdi_upper}).
 Since
 \begin{align*}
 \Bigl|F^{(i)}_{d}(D;s) \Bigr| \leq F^{(1)}_{d}(D;\sigma),
 \end{align*}
 it is enough to prove (\ref{Fdi_upper}) for $i=1$ and $s\in\mathbb{R}_{>d+2}$.
 Write  $F^{(1)}_{d}(D;s)$ as $A+B$ where
 \begin{align*}
  A
  &:=\sum_{D<t\le m/2}
   \frac{1}{t^{s-d-1} (m-t)}
   \sum_{m-t\le x_{1}\le\cdots\le x_{d}\le m}\frac{1}{x_{1}\cdots x_{d}}, \\
  B
  &:=\sum_{\substack{ D<t \\ m/2<t<m }}
   \frac{1}{t^{s-d-1} (m-t)}
   \sum_{m-t\le x_{1}\le\cdots\le x_{d}\le m}\frac{1}{x_{1}\cdots x_{d}}.
 \end{align*}
 Since
 \[
  \sum_{m-t\le x_{1}\le\cdots\le x_{d}\le m}\frac{1}{x_{1}\cdots x_{d}}
  \ll \frac{ t^d }{ (m-t)^d },
 \]
 we have
 \begin{align*}
  A
  &\ll \sum_{D<t\le m/2} \frac{1}{t^{s-2d-1} (m-t)^{d+1}}
  \ll \sum_{D<t\le m/2} \frac{1}{t^{s-2d-1} m^{d+1}} \\
  &\ll \sum_{D<t} \frac{1}{t^{s-d-1} }.
 \end{align*}
 We also have
 \begin{align*}
  B
  \ll \sum_{\substack{ D<t \\ m/2<t<m }} \frac{ (\log m)^{d} }{t^{s-d-1} (m-t)}
  \ll \sum_{ D<t<m<2t } \frac{ (\log m)^{d} }{t^{s-d-1} (m-t)}.
 \end{align*}
 Thus, putting $n=m-t$, we get
 \begin{align*}
  B
  \ll \sum_{\substack{ D<t \\ 1\le n<t }} \frac{ (\log(2t))^{d} }{t^{s-d-1} n}
  \ll \sum_{ D<t } \frac{ (\log(2t))^{d} \log t }{t^{s-d-1} }.
 \end{align*}
 This finishes the proof.
\end{proof}

We need Lemmas \ref{lem:Fd_formula} and \ref{lem:G_ab_formula} to prove Theorem \ref{main2}.
\begin{lem} \label{lem:Fd_formula}
 If $\sigma>d+2$, we have
 \[
  F_{d+1}(D;s)
  =F^{(1)}_{d}(D;s)-F^{(2)}_{d}(D;s)-F^{(3)}_{d}(D;s).
 \]
\end{lem}
\begin{proof}
 By the previous lemma, $F^{(1)}_{d}(D;s)$, $F^{(2)}_{d}(D;s)$, and $F^{(3)}_{d}(D;s)$ converge absolutely.
 Since
 \begin{align*}
  F^{(1)}_{d}(D;s)
  &=\sum_{D<t} \frac{1}{t^{s-d-1}} \sum_{0<x_{0}\le x_{1}\le\cdots\le x_{d}\le x_{0}+t}
   \frac{1}{x_{0}x_{1}\cdots x_{d}} \\
  &=\sum_{D<t} \frac{1}{t^{s-d-1}} \sum_{\substack{0<x_{0}\le x_{1}\le\cdots\le x_{d}\\x_{d}\le x_{0}+t}}
   \frac{1}{x_{0}x_{1}\cdots x_{d}}, \\
  F^{(2)}_{d}(D;s)
  &=\sum_{D<t} \frac{1}{t^{s-d-1}} \sum_{0<x_{d+1}-t\le x_{1}\le\cdots\le x_{d}\le x_{d+1}}
   \frac{1}{x_{1}\cdots x_{d}x_{d+1}} \\
  &=\sum_{D<t}\frac{1}{t^{s-d-1}}\sum_{\substack{0<x_{0}\le x_{1}\le\cdots\le x_{d}\\x_{d}>t\\x_{d}\le x_{0}+t}}
  \frac{1}{x_{0}x_{1}\cdots x_{d}}, \\ 
  F^{(3)}_{d}(D;s)
  &=\sum_{D<t<N}\frac{1}{N^{s-d-1}(N-t)}
   \sum_{N-t\le x_{1}\le\cdots\le x_{d}\le N}\frac{1}{x_{1}\cdots x_{d}} \\
  &=\sum_{D<N}\frac{1}{N^{s-d-1}}\sum_{\substack{0<x_{0}\le x_{1}\le\cdots\le x_{d}\le N\\x_{0}<N-D}
   }\frac{1}{x_{0}x_{1}\cdots x_{d}} \\
  &=\sum_{D<t}\frac{1}{t^{s-d-1}}\sum_{\substack{0<x_{0}\le\cdots\le x_{d}\\x_{d}\le t\\x_{0}<t-D}}
   \frac{1}{x_{0}\cdots x_{d}},
 \end{align*}
 we have
  \begin{align*}
   &F^{(1)}_{d}(D;s) -F^{(2)}_{d}(D;s) -F^{(3)}_{d}(D;s) \\
   &= \sum_{D<t}\frac{1}{t^{s-d-1}}\sum_{\substack{t-D\le x_{0}\le x_{1}\le\cdots\le x_{d}\le t}}
   \frac{1}{x_{0}x_{1}\cdots x_{d}}\\
  &=F_{d+1}(D;s).
 \end{align*}
 Hence we find the result.
\end{proof}

\begin{lem} \label{lem:G_ab_formula}
 Let $a$ be a non-negative integer and b an integer with $b\ge2$.
 If $\sigma>b,\sigma +\sigma_a>1+b,\ldots, \sigma+\sigma_a+\cdots+\sigma_1>a+b$, then
 $G_{a,b}(s_{1},\dots,s_{a};s)$ is well-defined and equal to
  \begin{align} \label{ppppp}
  \sum_{0<m_{1}<\cdots<m_{a}<m}\frac{1}{m_{1}^{s_{1}}\cdots m_{a}^{s_{a}}m^{s-b+1}}
  \sum_{m-m_{a}\le x_{1}\le\cdots\le x_{b-1}\le m}\frac{1}{x_{1}\cdots x_{b-1}}.
 \end{align}
 Here we understand that $m-m_a=m$ if $a=0$.
\end{lem}
\begin{proof}
 Note that \eqref{ppppp} converges absolutely.
 The proof is by induction on $b$.
 We can prove the case $b=2$ in the similar manner as in the proof of Theorem \ref{main1} for $\sigma>2$.
 For $b\ge3$, by the induction hypothesis, we have
 \begin{align*}
  & G_{a,b}(s_{1},\dots,s_{a};s) \\
  &=\sum_{n=0}^{\infty} G_{a+1,b-1}(s_{1},\dots,s_{a},s-n-b;n+b)
   -\sum_{n=0}^{\infty}G_{a+1,b-1}(s_{1},\dots,s_{a},-n;s+n) \\
  & =\sum_{n=0}^{\infty}\sum_{0<m_{1}<\cdots<m_{a+1}<m}
   \frac{1}{m_{1}^{s_{1}}\cdots m_{a}^{s_{a}} m_{a+1}^{s-n-b}m^{n+2}}
   \sum_{m-m_{a+1}\le x_{1}\le\cdots\le x_{b-2}\le m}
   \frac{1}{x_{1}\cdots x_{b-2}} \\
   &\quad -\sum_{n=0}^{\infty}\sum_{0<m_{1}<\cdots<m_{a+1}<m}
   \frac{1}{m_{1}^{s_{1}}\cdots m_{a}^{s_{a}} m_{a+1}^{-n}m^{s+n-b+2}}
   \sum_{m-m_{a+1}\le x_{1}\le\cdots\le x_{b-2}\le m}
   \frac{1}{x_{1}\cdots x_{b-2}}.
 \end{align*}
 Then we have
 \begin{align*}
  &\sum_{n=0}^{\infty}\sum_{0<m_{1}<\cdots<m_{a+1}<m}
   \frac{1}{m_{1}^{s_{1}}\cdots m_{a}^{s_{a}} m_{a+1}^{s-n-b}m^{n+2}}
   \sum_{m-m_{a+1}\le x_{1}\le\cdots\le x_{b-2}\le m}
   \frac{1}{x_{1}\cdots x_{b-2}} \\
  &=\sum_{0<m_{1}<\cdots<m_{a+1}<m}
   \frac{1}{m_{1}^{s_{1}}\cdots m_{a}^{s_{a}} m_{a+1}^{s-b+1} } \biggl(\frac{1}{m-m_{a+1}} -\frac{1}{m} \biggr)
   \sum_{m-m_{a+1}\le x_{1}\le\cdots\le x_{b-2}\le m}
   \frac{1}{x_{1}\cdots x_{b-2}} \\
  &=\sum_{0<m_{1}<\cdots<m_{a}}
   \frac{1}{m_{1}^{s_{1}}\cdots m_{a}^{s_{a}} }
   \Bigl( F^{(1)}_{b-2} (m_a;s) -F^{(2)}_{b-2} (m_a;s) \Bigr)
\end{align*}
 and
 \begin{align*}
   &\sum_{n=0}^{\infty}\sum_{0<m_{1}<\cdots<m_{a+1}<m}
   \frac{1}{m_{1}^{s_{1}}\cdots m_{a}^{s_{a}} m_{a+1}^{-n}m^{s+n-b+2}}
   \sum_{m-m_{a+1}\le x_{1}\le\cdots\le x_{b-2}\le m}
   \frac{1}{x_{1}\cdots x_{b-2}} \\
   &=\sum_{0<m_{1}<\cdots<m_{a+1}<m}
   \frac{1}{m_{1}^{s_{1}}\cdots m_{a}^{s_{a}} m^{s-b+1} (m-m_{a+1})}
   \sum_{m-m_{a+1}\le x_{1}\le\cdots\le x_{b-2}\le m}
   \frac{1}{x_{1}\cdots x_{b-2}} \\
  &=\sum_{0<m_{1}<\cdots<m_{a}}
   \frac{1}{m_{1}^{s_{1}}\cdots m_{a}^{s_{a}} }
   F^{(3)}_{b-2} (m_a;s).
\end{align*}
 Since the series
 \begin{align*}
  &\sum_{0<m_{1}<\cdots<m_{a}}
   \frac{1}{m_{1}^{s_{1}}\cdots m_{a}^{s_{a}} }
   F^{(i)}_{b-2} (m_a;s) \qquad (i=1,2,3)
 \end{align*}
 are absolutely convergent by Lemma \ref{3232}, we see that $G_{a,b}(s_{1},\dots,s_{a};s)$ is also absolutely convergent, and furthermore, we have
 \begin{align*}
  & G_{a,b}(s_{1},\dots,s_{a};s) \\
  & =\sum_{0<m_{1}<\cdots<m_{a}}\frac{1}{m_{1}^{s_{1}}\cdots m_{a}^{s_{a}}}
   \Bigl( F^{(1)}_{b-2}(m_{a};s) -F^{(2)}_{b-2}(m_{a};s) -F^{(3)}_{b-2}(m_{a};s) \Bigr) \\
  & =\sum_{0<m_{1}<\cdots<m_{a}}\frac{1}{m_{1}^{s_{1}}\cdots m_{a}^{s_{a}}}F_{b-1}(m_{a};s) \\
  & =\sum_{0<m_{1}<\cdots<m_{a}<m}\frac{1}{m_{1}^{s_{1}}\cdots m_{a}^{s_{a}}m^{s-b+1}}
   \sum_{m-m_{a}\le x_{1}\le\cdots\le x_{b-1}\le m}\frac{1}{x_{1}\cdots x_{b-1}}
 \end{align*}
 from Lemma \ref{lem:Fd_formula}.
 Hence the lemma is proved.
\end{proof}
%

\begin{proof}[Proof of Theorem \ref{main2}]
 From Lemma \ref{lem:G_ab_formula}, we have
 \begin{align*}
  G_{0,b}(s)
  &= \sum_{0<m}\frac{1}{m^{s-b+1}}\sum_{m\le x_{1}\le\cdots\le x_{b-1}\le m}\frac{1}{x_{1}\cdots x_{b-1}}\\
  &= \zeta(s).\qedhere
 \end{align*}
\end{proof}

\section*{Acknowledgements}
This work was supported by JSPS KAKENHI Grant Numbers JP18J00982, JP18K13392, and JP19K14511.


\end{document}